\documentclass[11pt]{article}
 \usepackage[top=1in,bottom=1in,left=1.5in,right=1.5in]{geometry}
  \usepackage{amsmath,amssymb}
  \usepackage{latexsym}% defines $\Box$ for LaTeX2e
  \usepackage{setspace}
  \usepackage{tikz}

  \newcommand{\cc}{\mathcal{C}}

  \newcommand{\hs}{\hspace*{\parindent}}

  \newcommand{\qed}{\hspace*{\fill} $\Box$\\}

  \newtheorem{theo}{\bfseries \hs Theorem}[section]

  \newtheorem{lemma}[theo]{\bfseries \hs Lemma}

  \numberwithin{equation}{section} % Automatically number equations within sections

\begin{document}

\title{Almost-rainbow edge-colorings of some small subgraphs}

\author{Elliot Krop\thanks{Department of Mathematics, Clayton State University, (ElliotKrop@clayton.edu)} \and Irina Krop\thanks{Depaul University, (irina.krop@gmail.com)}}
   \date{\today}

 \maketitle

\begin {abstract}
Let $f(n,p,q)$ be the minimum number of colors necessary to color the edges of $K_n$ so that every $K_p$ is at least $q$-colored.  We improve current bounds on these nearly ``anti-Ramsey" numbers, first studied by Erd\H os and Gy\'arf\'as.  We show that $f(n,5,9) \geq \frac{7}{4}n-3$, slightly improving the bound of Axenovich. We make small improvements on bounds of Erd\H os and Gy\'arf\'as by showing $\frac{5}{6}n+1\leq f(n,4,5)$ and for all even $n\not\equiv 1 \pmod 3$, $f(n,4,5)\leq n-1$ . For a complete bipartite graph $G=K_{n,n}$, we show an n-color construction to color the edges of $G$ so that every $C_4\subseteq G$ is colored by at least three colors.  This improves the best known upper bound of M. Axenovich, Z. F\"uredi, and D. Mubayi. 
\\[\baselineskip] 2000 Mathematics Subject
      Classification: 05A15, 05C38, 05C55
\\[\baselineskip]
\end {abstract}

 \section{Introduction}
 
 \subsection{Definitions}
 
 For basic graph theoretic notation and definition see Diestel \cite{Diest}.  All graphs $G$ are undirected with the vertex set $V$ and edge set $E$.  We use $\left|G\right|$ for $\left|V\right|$ and $\left\|G\right\|$ for $\left|E\right|$.  $K_n$ denotes the complete graph on $n$ vertices and $K_{n,m}$ the bipartite graph with $n$ vertices and $m$ vertices in the first and second part, respectively.  For any edge $(u,v)$, let $\cc(u,v)$ be the color on that edge, and for any vertex $v$, let $\cc(v)$ be the set of colors on the edges incident to $v$.  We say that an edge-coloring is $proper$ if every pair of incident edges are of different colors. If vertices $u,v$ are adjacent, we write $u \sim v$.

 \subsection{Coloring Edges}
 
 Given a graph $G$ of order $n$ and integers $p,q$ so that $2\leq p \leq n$ and $1 \leq q \leq {p \choose 2}$, call an edge-coloring $(p,q)$ if every $K_p\subseteq K_n$ receives at least $q$ colors on its edges.
 Let $f(n,p,q)$ be the minimum colors in a $(p,q)$ coloring of $K_n$.  This generalization of classical Ramsey functions was first mentioned by P. Erd\H os in $\cite {Erd}$ and later studied by Erd\H os and Gy\'arf\'as in $\cite {ErdGyar}$.  Further, define $\phi(n,p,q)$ to be the minimum colors in a proper $(p,q)$ coloring of $K_n$.
 
 Extending the definition, for any graph $G$, call an edge coloring $(H,q)$ if every subgraph $H \subseteq G$ receives at least q colors on its edges.  Let $f(G,H,q)$ be the minimum colors in an $(H,q)$ coloring of the edges of $G$.  We say that a coloring of $H$ is \emph{almost-rainbow} if $q = \left\|H\right\| - 1$, that is, one color is repeated once.
 
 For an extended survey regarding bounds on rainbow colorings, see \cite{FujMagOze}.
 
 Using the Local Lemma, the authors in $\cite{ErdGyar}$ were able to produce bounds for $f(n,p,q)$, with several difficult cases unresolved.  Among those were $f(n,4,3)$, $f(n,4,4)$, $f(n,4,5)$, and $f(n,5,9)$.  In these cases they showed that $f(n,4,3) \leq c \sqrt{n}$, $c\sqrt{n}\leq f(n,4,4)\leq cn^{\frac{2}{3}}$, $\frac{5n-1}{6}\leq f(n,4,5) \leq n$, and $\frac{4}{3} n \leq f(n,5,9) \leq cn^{\frac{3}{2}}$. The authors further mentioned that in this branch of generalized Ramsey theory, finding the orders of magnitude of $f(n,4,4)$ and $f(n,5,9)$ are ``the most interesting open problems, at least to show that the latter is non-linear".  The authors then stated the linearity of said function as Problem 1.
  
  As for $f(n,4,5)$, the authors showed that $\frac{5(n-1)}{6}\leq f(n,4,5)$ with an upper bound of $n$ for odd $n$ and $n-1$ for even $n$ if $n-1$ is prime.
  
 In $\cite{Mub}$, D.Mubayi showed that \[f(n,4,3) \leq e^{O(\sqrt{log n})}\] and in $\cite{KostMub}$ A.Kostochka and D.Mubayi showed that for some constant c, \[f(n,4,3) \geq \frac{c\log n}{\log \log \log n}.\] J.Fox and B. Sudakov in \cite{FoxSud}, further improved the lower bound to $\frac{\log n}{4000}$.\\ 
 
 As for the other case, in $\cite {Axe}$, M.Axenovich showed that for some constant c, \[\frac{1+\sqrt{5}}{2}n \leq f(n,5,9) \leq 2n^{1+\frac{c}{\sqrt{\log n}}}.\]
 In that same paper, she remarked that G.T\'oth had communicated to her that the lower bound can be improved to $2n-6$, however, the result has remained unpublished for over ten years.
 
 In Section 2, we show
 
 \[f(n,5,9) \geq \frac{7}{4} n - 3\]
 
 \vspace{.2 in}
 
 In Section 3, we make minimal improvements in the work of \cite{ErdGyar}, showing $\frac{5}{6}(n-1)+1\leq f(n,4,5)\leq n-1$ for even $n$ not congruent to one mod three.
 
 \vspace{.2 in}
 
 In $\cite{AxeFurMub}$, the authors showed that $f(K_{n,n},C_4,3) \geq \frac{2}{3}n$, $f(K_{n,n},C_4,3) \leq n$ for odd $n\geq5$, and $f(K_{n,n},C_4,3) \leq n+1$ for even $n\geq5$.
 
 In Section 4, we show
 
 \[f(K_{n,n},C_4,3) \leq n, \mbox{ for all }n\geq3.\]

 We believe that this upper-bound is the best possible.
 
 \section{Almost rainbow five-cliques}
 
 \subsection{The main tool}
 
 Let $f(G)$ be the minimum number of colors needed to color the edges of $G$ so that every path or cycle with four edges is at least three-colored.  
  
 Let $\phi(G)$ be defined as $f(G)$ above, except replace ``color" by ``properly color".
 By arguments from $\cite{Axe}$ it is easy to see that $f(n,5,9) \leq \phi(n,5,9) = \phi(K_n)$.
 
 %Set $G = K_{2,n}$ and integer $n \geq 2$.  Call the vertices in the first part of $G$, $u$ and $v$.  We properly color the edges of $G$ so that every path or cycle with four edges is at least three-colored.  %Notice that the minimum number of colors for such a coloring is $f(K_{2,n})$ and that $f(K_{2,n},K_{2,3},4) = f(K_{2,n})$. %Indeed, given $K_{2,3}$ with five colors on its six edges, one color must be repeated. Consequently, every path or cycle with four edges can have at most one repeated color, therefore receiving at least three colors. On the other hand, every $K_{2,3}$ containing paths or cycles with four edges colored by at most three colors, must receive at least four colors.

\begin {lemma}\label{maintool}

$\phi(K_{2,n}) = \lceil \frac{3}{2} n \rceil$

\end {lemma}

\begin {proof}
 
  Suppose the edges of $G = K_{2,n}$ are properly colored so that every path of length four receives at least three colors.  Call the vertices in the first part of $G$, $u$ and $v$. Choose a color $a \in \cc(u)\cap \cc(v)$ so that for some vertices $x,y$ in the second part of $G$, $a=\cc(u,x)=\cc(v,y)$. Note that there exist colors $b,c$ so that $b=\cc(u,y)$, $c=\cc(v,x)$, and $b,c \in (\cc(u) \cup \cc(v)) \backslash (\cc(u) \cap \cc(v))$. Since there are two colors for every one in $\cc(u)\cap \cc(v)$, we can say that

 \begin {equation}\label{K2ntwo-int}
   \left|\cc(u) \cap \cc(v) \right| \leq \lfloor \frac {1}{2} \left| (\cc(u) \cup \cc(v)) \backslash (\cc(u) \cap \cc(v)) \right| \rfloor.
 \end {equation}

Applying this inequality to the principle of inclusion-exclusion, we write

\[ \left|\cc(u) \cup \cc(v) \right| = \left| \cc(u) \right| + \left| \cc(v) \right| - \left| \cc(u) \cap \cc(v) \right| \geq 2n - \frac{1}{3} \left| \cc(u) \cup \cc(v) \right.| \]

Solving for the union we get
\begin {equation}\label{ubK2n}
\left|\cc(u) \cup \cc(v) \right| \geq \frac{3}{2} n. 
\end {equation}

%Next suppose that the coloring is not proper and let $a$ and $b$ be colors in $c(u)$.  If $a$ is used on incident edges adjacent to $u$ and $b$ is used on incident edges adjacent to $u$, then consider the induced subgraph composed of the vertices incident to any four edges colored by $a$ or $b$.  Notice that the maximum number of colors used on the edges of this subgraph is six.  Hence, the maximum number of times that a color can be repeated in $c(u)$ is once.

%Assume $a \in c(u)$ is repeated once, no two color in $c(v)$ are the same, and let $v_1,v_2$ be vertices incident to the edges colored $a$.  Notice that the edge coloring on $H := G \backslash \{v_1,v_2\}$ is proper.  Applying the above result for proper colorings, we conclude that the number of colors on the edges of $H$ is at least $\frac{3}{2} (n-2)$.  Since $c(v_1,v) \neq c(v_2,v)$, the number of colors on the edges of $G$ is at least $\frac{3}{2} (n-2) + 3 = \frac{3}{2} n$ as desired.

%Last, suppose that $a \in c(u)$ is repeated once and $b \in c(v)$ is repeated once.  Let $v_1,v_2$ be vertices incident to the edges colored $a$ and $v_3,v_4$ be vertices incident to the edges colored $b$.  Notice that the edge coloring on $H := G \backslash {v_1,v_2,v_3,v_4}$ is proper.  Applying the result for proper colorings, we conclude that the number of colors on the edges of $H$ is at least $\frac{3}{2} (n-4)$.  Since $c(v_1,v) \neq c(v_2,v)$, and $c(v_3,u) \neq c(v_4,u)$, the number of colors on the edges of $G$ is at least $\frac{3}{2} (n-4) + 6 = \frac{3}{2} n$ concluding the proof of the lower bound.\\

For the upper bound, we construct an edge-coloring of $G=K_{2,n}$ with $\lceil\frac{3}{2} n\rceil$ colors.  Label the vertices of the first part of $G$, $u,v$ and the second part $\{v_1,v_2, \dots, v_n\}$.  Let $r = \lceil \frac{n}{2} \rceil$.  Color the edges $(v_1,u), (v_2,u), \dots, (v_r,u)$ by the colors $1, \dots, r$. If $n$ is even, color the edges $(v_n,v), (v_{n-1},v), \dots, (v_{n-r+1},v)$ from the set of colors $\{1, \dots, r\}$. If $n$ is odd, color the edges $(v_n,v), (v_{n-1},v), \dots, (v_{n-r+2},v)$ by some of the colors from the set $\{1, \dots, r\}$. Color the remaining edges distinctly by all the colors not previously used.  Let $i$ and $j$ be such that $\cc(u,v_i) = \cc(v,u_j)$.  Notice that for any $k \in \{1, \dots, n\}$, $\{\cc(u,v_i), \cc(u,v_j), \cc(v,v_i), \cc(u,v_k)\}$ are pairwise distinct. Hence every 4-path receives at least three colors.

\end {proof}
\qed

 \subsection{A small improvement}
 
\begin {theo}

 $f(n,5,9) \geq \frac{7}{4} n - 3$
 
\end {theo}

\begin {proof}
 
 Consider a $(5,9)$ edge-coloring of $G=K_n$ using $s$ colors. Using the argument of M. Axenovich \cite{Axe}, we first assume that the coloring is not proper, so there exist incident edges $(v_1,v_2)$ and $(v_1,v_3)$ of the same color.  For the coloring to remain $(5,9)$, all edges of $G\backslash\{(v_1,v_2),(v_1,v_3)\}$ incident to $\{v_1,v_2,v_3\}$ must be of different colors and not $\cc(v_1,v_2)$ or $\cc(v_2,v_3)$.  Therefore, $s \geq 3n-7 \geq \frac{7}{4} n - 3$ for $n \geq 5$.
 
 Next we assume the coloring is proper.  By the pigeonhole principle there exists a color, call it $a$, used on at least ${n \choose 2} \slash s$ edges.  Let $A$ be the set of vertices adjacent to edges colored $a$ and choose vertices $u,v \in A$ so that $c(u,v)=a$.
 
 We say that an edge is $in$ $A$ if both vertices adjacent to that edge are in $A$.  Notice that the number of colors on the edges in $A$ adjacent to $u$ $\geq 2 {n \choose 2} \slash s - 1$, the same for $v$, and $c(u,v)$ is counted both times.  Let $H$ be the complete bipartite graph with vertices $\{u,v\}$ in the first part and the vertices of $G \backslash A$ in the second part.  Let the edge coloring of $H$ be induced by the edge coloring of $G$.  For any $x \in A$ and $y \in G$, $\cc(u,x) \neq \cc(v,y)$, else we produce a two-colored four-edge path.  The same reasoning holds for $y \in A$ and $x \in G$.  This implies that the colors on the edges of $H$ are distinct from the colors previously counted.  Hence we apply Lemma \ref{maintool} to $H$ to obtain
 
 \begin{equation}\label{count}
  s \geq 2 \frac{{n \choose 2}}{s} - 1 + 2 \frac{{n \choose 2}}{s} - 1 - 1 + \frac{3}{2} (n - 2 \frac{{n \choose 2}}{s}).
  \end{equation}
 
 Solving for s we obtain the result.

\end {proof}
\qed

 \section{Almost rainbow four-cliques}
 
 We obtain a marginal improvement for the lower bound on $f(n,4,5)$ and extend the even case of the upper bound from \cite{ErdGyar} to all complete graphs with orders not congruent to one modulo three.
 
 \begin{theo}\label{4,5}
 \mbox{}
\begin{enumerate}
	\item $\frac{5}{6}(n-1)+1 \leq f(n,4,5)$
	\item $f(n,4,5)\leq n-1$ for even $n \not \equiv 1 \pmod 3$
 \end{enumerate}

 \end{theo}
 
 \begin{proof}
 
 Given a $(4,5)$ coloring of the edges of $G=K_n$, for a fixed vertex $u$, let $P_u$ denote the set of edges incident to $u$, whose colors are repeated on other edges incident to $u$.  Let $S_u$ denote the set of edges with non-repeated colors, incident to $u$.  Let $T_u$ denote the set of edges incident to edges from $P_u$ of the same color.
 
 
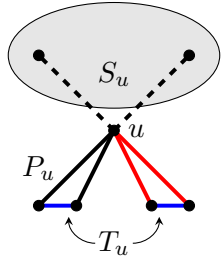
\begin{center}

\begin{tikzpicture}[>=stealth]

\filldraw[fill=gray!20!white] (0,1) ellipse (40pt and 20pt);

\draw [line width=1.5, style=dashed, color=black] (0,0) -- (1,1);
\draw [line width=1.5, style=dashed, color=black] (0,0) -- (-1,1);
\draw [line width=1.5, color=red] (0,0) -- (1,-1);
\draw [line width=1.5, color=red] (0,0) -- (.5,-1);
\draw [line width=1.5, color=black] (0,0) -- (-1,-1);
\draw [line width=1.5, color=black] (0,0) -- (-.5,-1);
\draw [line width=1.5, color=blue] (1,-1) -- (.5,-1);
\draw [line width=1.5, color=blue] (-1,-1) -- (-.5,-1);

\filldraw
 (0,0) circle (2pt)
 (.5,-1) circle (2pt)
 (-.5,-1) circle (2pt)
 (1,1) circle (2pt)
 (-1,1) circle (2pt)
 (1,-1) circle (2pt)
 (-1,-1) circle (2pt);
 
\node at (.3,0)
{$u$};

\node at (0,.75)
{$S_u$};

\node at (-1,-.5)
{$P_u$};

\node at (0,-1.5)
{$T_u$};

\draw [->] (.25,-1.5) arc (270:360:10pt);
\draw [->] (-.25,-1.5) arc (270:180:10pt);

\end{tikzpicture}

\end{center}
 
 Notice that 
 
\begin{enumerate}
 \item $\cc(P_u)\cap \cc(S_u)=\emptyset$ by definition  
 \item $\cc(P_u)\cap \cc(T_u)=\emptyset$ else we obtain an induced four-colored $K_4$ on the edges $p\in P_u$ and $t\in T_u$ that share the same same color and the edges $p_1,p_2 \in P_u$ that share the same color and are incident to $t$ ($p$ may be equal to $p_1$, depending on the coloring).
 \item $\cc(S_u)\cap \cc(T_u)=\emptyset$ else we obtain an induced four colored $K_4$ on the the edge $s\in S_u$ and $t\in T_u$ of the same color and the two edges of $P_u$ with the same color, which are incident to $t$
 \item For any vertex $v$ distinct from $u$, if $(u,v)\in P_u$ so that $\cc(u,v)=\cc(u,w)$ for some $w$, then $(u,v)\notin P_v$ and $(v,w)\notin P_v$
 \item For any vertex $v$ distinct from $u$, $T_u \cap T_v = \emptyset$
 %\item For any vertex $v$ distinct from $u$, $\cc(P_u) \cap \cc(P_v)=\emptyset$ since no vertex can be incident to more than two edges colored by the same color, else the induced subgraph would be a four-colored $K_4$
\end{enumerate}

Notice that \[2\sum_{u\in V(G)}{|T_u|} = \sum_{u\in V(G)}{|P_u|}\]

so that \[\sum_u{|T_u|}+\sum_u{|P_u|}=3\sum_u{|T_u|}=3\frac{1}{n}\sum_u{|T_u|}\times n \leq {n\choose 2}\]

by the above claim $5$, and we obtain \[\frac{1}{n}\sum_u{|T_u|}\leq \frac{n-1}{6}.\]
 
By the pigeonhole principle, choose a vertex $u$ so that $|T_u|\leq \frac{n-1}{6}$. Notice that $n-1=\deg{u}=|S_u|+|P_u| \leq |S_u| + \frac{n-1}{3}$, so that \[|S_u| \geq \frac{2}{3}(n-1)\]
 
Summing up the colors of edges incident to $u$ we get \[|\cc(u)| = |S_u| + \frac{1}{2}|P_u| \geq \frac{2}{3}(n-1) + \frac{1}{6}(n-1) = \frac{5}{6}(n-1).\]

However, $\cc(T_u)$ must be nonempty and distinct from the colors counted above, hence

\[|\cc(u)|\geq \frac{5}{6}(n-1)+1.\]

\vspace{.2 in}

For the upper bound we color the edges of $K_n$ by a classical proper coloring (see \cite{Wilson} for example) and show that such a coloring is $(4,5)$.
 
 For odd $n$, we n-color the edges of $K_n$ by drawing the vertices in the form of a regular n-gon and coloring the consecutive edges around the boundary in order with colors $1$ to $n$. Next we color every edge parallel to a boundary edge by the same color as that boundary edge. Call the resulting labeled graph $G_n$. Notice that every $K_4 \subseteq G_n$ with a pair of parallel edges is a non-rectangular trapezoid. Hence the coloring is $(4,5)$.
 
 For even $n$, choose a $K_{n-1}$ subgraph and color it as above, obtaining $G_{n-1}$. Next construct the graph $w \times G_{n-1}$, joining the above graph to a vertex $w$. Since for any vertex $u$ of $G_{n-1}$, there are only $n-2$ incident edges, some color is missing. Apply this color to the edge $(u,w)$ and continue likewise for all vertices of $G_{n-1}$. Call the resulting labeled graph $G^*_n$.
 
 For vertices $x, y, z \in G^*_n$ with so that $(x,y)$ and $(y,z)$ are boundary edges, we say that $y$ is \emph{opposite} an edge $e$ if the line bisecting angle $uvw$ is the perpendicular bisector of $e$. Notice that the edges opposite to $y$ share the same color, which is not used on any edge incident to $y$. By the above observation, $G_{n-1} \subseteq G^*_n$ is $(4,5)$-colored, hence it is enough to show that for $w$ as chosen above in the definition of $G^*_n$ and any other distinct vertices $x,y,z$ of $G^*_n$, the induced subgraph receives at most one repeated color. %Let $v$ be the non-$w$ endvertex of the edge incident to $w$ colored $1$. 
 Choose any vertex $v\in G^*_n$. For $i = 1, \dots, n-2$ label the vertices with counterclockwise distance $i$ from $v$, $u_i$, where arithmetic of label indices is performed modulo $n-1$. %Similarly, $i = 1, \dots, n-1$ label the vertices with clockwise distance $i$ from $v$, $v_i$. 
 Notice that the only edges that share the color $\cc(w,v)$ are $(u_1,u_{-1}), (u_2,u_{-2}), \dots, (u_{n-2},u_{-(n-2)})$. For $i = 1, \dots, \frac{n-2}{2}$, if $\cc(u_i,w)=\cc(u_{-i},v)$, then for any edge $e$ opposite $u_i$, $\cc(e)=\cc(u_{-i},v)$. However, this means that 
 
 \[\cc(u_{i-1},u_{i+1})=\cc(e)=\cc(u_{-i},v) \Leftrightarrow vu_{2k}=vu_{-k} \Leftrightarrow\]
 \[3k\equiv0 \pmod{(n-1)} \Leftrightarrow n \equiv 1 \pmod 3.\]

%Choose another vertex $v$ and observe that if $t_1 \in T_u$, $t_2 \in T_v$, and $c(t_1)=c(t_2)$, then $t_1$ is not incident to $t_2$, else we would obtain an induced four-colored $K_4$ on $t_1,t_2$ and the two edges of the same color incident to $t_1$ (or $t_2$).

%Define $H=\bigcup_{v\in V(G)}T_v$. By the previous observation no pair of incident edges in $H$ can be of the same color. By Vizing's theorem, $\Delta(H)\leq \chi'(H) \leq \Delta(H)+1$.

%Let $r=\min_{v \in V(G)}|T_v|$. Notice that the number of colors on the edges of $G$ is at least $n-1-r$. Also, $\left\|H\right\|\geq rn$, so we can find a vertex $w$ with degree at least $2r$ in $H$.

 \qed \end{proof}

  \section{Almost-Rainbow Four-Cycles}
 We show the improved upper bound for the bipartite problem, when the two parts of $G$ are of equal size.
 
 \begin {theo}
 
 \[f(K_{n,n},G_4,3) \leq n, \mbox{ for all }n\geq 3\]

 \end {theo}

 \subsection {The Coloring}
\vspace{.1 in}

 We will explore the matrix
 
 $$G=\begin{pmatrix}
1           &2        &3         &. &r         &. &c+1        &.     &n  \cr
3           &1        &2         &. &r-1       &. &c          &.     &n-1 \cr
v_{3}     &n-1      &1         &. &r-2       &. &c-1        &.     &n-2 \cr
.           &.        &.         &. &.         &. &.          &.     &.   \cr
v_{n+1-r} &r+1      &r+2       &. &1         &. &r+c        &.     &r    \cr
.           &.        &.         &. &.         &. &.          &.     &.    \cr
v_{n-1}   &3        &4         &. &r+1       &. &c+2        &.     &2    \cr
n-2         &u_{2}  &u_{3}   &. &u_{r} &. &u_{c+1}  &.     &1
\end{pmatrix}$$

The values of $v_i$ and $u_i$ will be defined shortly.

\vspace{.1 in}
 
Let permutation $\sigma$ be the $n-1$ cycle $(1\; 2\; \ldots\;  n-1)$. That is, $\sigma$ sends $i$ to $i+1 \pmod{n-1}$. For a natural number $m$ we shall write $m \pmod {(n-1)}$ for its representative in $\{1,2,\ldots ,n-1\}$. For each $r$ we defined $\sigma^{(r)}$ by the rule $\sigma^{(r)}{(c)} \equiv r+c \pmod{(n-1)}$.  Let us start with the matrix

$$C=\begin{pmatrix}
2                   &3                      &. &c+1                 &.       &n                   \cr
\sigma^{0}{(1)}   &\sigma^{0}{(2)}      &. &\sigma^{0}{(c)}   &.       &\sigma^{0}{(n-1)} \cr
\sigma^{n-2}{(1)} &\sigma^{n-2}{(2)}    &. &\sigma^{n-2}{(c)} &.       &\sigma^{n-2}{(n-1)}\cr
.                   &.                      &. &.                   &.       &. \cr
\sigma^{r}{(1)}   &\sigma^{r}{(2)}      &. &\sigma^{r}{(c)}   &.       &\sigma^{r}{(n-1)}\cr
.                   &.                      &. &.                   &.       &.  \cr
\sigma^{2}{(1)}   &\sigma^{2}{(2)}      &.   &\sigma^{2}{(c)} &.      &\sigma^{2}{(n-1)}
\end{pmatrix}$$

\vspace{.1 in}

We define matrix $G$ by adding the first column $V=\{v_1,\dots,v_{n-1},vu\}$ and the last row $U=\{vu,u_2,\dots,u_n\}$ to the matrix $C$.

$$G=\begin{pmatrix}
v_{1}       &2                   &3                      &. &c+1                 &.       &n                   \cr
v_{2}       &\sigma^{0}{(1)}     &\sigma^{0}{(2)}        &. &\sigma^{0}{(c)}     &.       &\sigma^{0}{(n-1)} \cr
v_{3}       &\sigma^{n-2}{(1)}   &\sigma^{n-2}{(2)}      &. &\sigma^{n-2}{(c)}   &.       &\sigma^{n-2}{(n-1)}\cr
.             &.                   &.                      &. &.                   &.       &. \cr
v_{n+1-r}   &\sigma^{r}{(1)}     &\sigma^{r}{(2)}        &. &\sigma^{r}{(c)}     &.       &\sigma^{r}{(n-1)}\cr
.             &.                   &.                      &. &.                   &.       &.  \cr
v_{n-1}     &\sigma^{2}{(1)}    &\sigma^{2}{(2)}         &. &\sigma^{2}{(c)}     &.       &\sigma^{2}{(n-1)}\cr
vu            &u_{2}                &u_{3}                  &. &u_{c+1}           &.       &u_{n}
\end{pmatrix}$$

\vspace{.1 in}

The entries of $G$ will be defined as follows:
for every 4-tuple $(i,j;l,m)$ with $1 \leq i<j \leq n$ and $1\leq l<m \leq n$ the  ($2\times 2$) matrix

$$G(i,j;l,m)=\begin{pmatrix}
                    a_{il}&a_{im}\cr
                    a_{jl}&a_{jm}\end{pmatrix}$$

\vspace{.1 in}
 
 We consider the colorings for the edges $V$ and $U$ in three types of even $n\pmod 6$.

\vspace{.1 in}
 
 \noindent \emph{\textbf{Type 1:} Matrix $G_1=G$ for $n \equiv 2 \pmod 6$};    $ \;  \; \left[ \;  n=2+6k, \; k\geq 1 \; \right]$

   \[ a_{i,1} = \left\{\begin{array}{ll}
    1, & i=1\\
    3, & i=2\\
    n, & 3\leq i \leq \frac{n}{2} +1\\
    2(i-1)-n, & \frac{n}{2}+2\leq i\leq n-1\\ 
    n-2, & i=n  
    \end{array}\right. \]

    \[ a_{n,l} = \left\{\begin{array}{ll}
    n-2l, & 1\leq l \leq \frac{n}{2}-1\\
    n, & \frac{n}{2} \leq l \leq n-2\\
    n-1, & l=n-1\\
    1, & l=n    
    \end{array}\right. \]
 
     \vspace{.1 in}
     
\noindent \emph{\textbf{Type 2:} Matrix $G_2=G$ for $n \equiv 6 \pmod 6$};  $ \; \; \left[ \; n=6+6k, \; k\geq 1 \; \right]$

 \noindent  We define $Y$ as $\frac{n}{2}-2$ for even $k$,  and as $\frac{n}{2}+1$ for odd $k$.
   
   \[ a_{i,1} = \left\{\begin{array}{ll}
    1, & i=1\\
    3, & i=2\\
    n, & 3\leq i \leq \frac{n}{2} +1\\
    Y, & i=\frac{n}{2}+2\\
    2(i-2)-n, & \frac{n}{2}+3\leq i\leq n-1\\
    n-2, & i=n   
    \end{array}\right. \]

   \[ a_{n,l} = \left\{\begin{array}{ll}
    n-2, & l=1\\
    n-2(l+1), & 2\leq l \leq \frac{n}{2}-2\\
    Y, & l= \frac{n}{2}-1\\
    n, & \frac{n}{2} \leq l \leq n-2\\
    n-1, & l=n-1\\
    1, & l=n    
    \end{array}\right. \]
 \vspace{.1 in}
 
 \noindent  Exception for $n=6$;  $\left[ k= 0 \right]$ the first row $V=\{1,5,6,6,4,\}$, the last column $U=\{3,6,6,6,5,1\}$.   
    
   \vspace{.1 in}

 \noindent \emph{\textbf{Type 3:} Matrix $G_3=G$ for $n \equiv 4 \pmod 6$};  $ \; \; \left[ \; n=4+6k, \; k\geq 4 \; \right]$

  \noindent  The regularity  starts with $n>22$.  
  
      \[ a_{i,1} = \left\{\begin{array}{ll}
    1, & i=1\\
    3, & i=2\\
    n, & 3\leq i \leq \frac{n}{2} +1\\
    n-9, & i=\frac{n}{2}+2\\
    2(i-2)-n, & \frac{n}{2}+3\leq i\leq \frac{5n+4}{6}\\
    2(i-1)-n, & \frac{5n+10}{6}\leq i\leq n-1\\
    n-2, & i=n   
    \end{array}\right. \]

     \[ a_{n,l} = \left\{\begin{array}{ll}
    n-2l, & 1\leq l \leq \frac{n-4}{6}\\
    n-2(l+1), & \frac{n+2}{6}\leq l \leq \frac{n}{2}-2\\
    n-9, & l=\frac{n}{2}-1\\
    n, & \frac{n}{2} \leq l \leq n-2\\
    n-1, & l=n-1\\
    1, & l=n    
    \end{array}\right. \]

 \vspace{.1 in}
 
 \noindent  Exceptions:

 For $n=10$ we replace $(n-9)$ with $(n-8)$.

 For $n=16$ we replace $(n-9)$ with $(n-11)$.

 For $n=22$ we replace 
 $(n-9)$ with $(n-5)$ and the definitions:
 
  \[ a_{i,1} = \left\{\begin{array}{ll}
    2(i-2)-n, & \frac{n}{2}+3\leq i\leq \frac{5n-2}{6}\\
    2(i-1)-n, & \frac{5n+4}{6}\leq i\leq n-1
    \end{array}\right. \]

    \[a_{n,l} =  \left\{\begin{array}{ll}   
     n-2l, & 1\leq l \leq \frac{n-10}{6}\\
     n-2(l+1), & \frac{n-4}{6}\leq l \leq \frac{n}{2}-2
    \end{array}\right. \]

 \vspace{.1 in}
 
 \subsection{Sketch of Proof}
 
 \begin{proof}
 
  First, we show that every $4$-cycle defined in the \emph{basic coloring} (matrix entries $a_{ij}$ where $1< i \leq n, 1\leq j < n$) is $almost$ $rainbow$.  That is, given $i<j$ and $l<m$ we show that $a_{i,l},a_{j,l},a_{i,m},a_{j,m}$ contains at least three distinct elements in the \emph{basic coloring}.
  
\vspace{.1 in}

\noindent \emph{\textbf{Step 1}}:

We start with the matrix $C$ and look at two occurrences, which are identical for each of the types of even $n\pmod 6$ specified above.
\vspace{.1 in}

\emph{\textbf{Case 1}}:
We take the submatrix of $G(i,j;l,m)$ with $2 \leq l < m \leq n$, $2 \leq i <j < n$, and let $s=(n+1)-i$, $t=(n+1)-j$. 
A typical ($2\times 2$) submatrix in this case has the form:

$$\begin{pmatrix}
                    \sigma^{s}{(l-1)}& \sigma^{s}{(m-1)}\cr
                    \sigma^{t}{(l-1)}& \sigma^{t}{(m-1)}\end{pmatrix}$$

We wish to show there are three distinct elements:
$\sigma^{s}{(l-1)}\neq \sigma^{t}{(l-1)}$,
$\sigma^{t}{(l-1)}\neq \sigma^{t}{(m-1)}$, 
$\sigma^{s}{(l-1)}\neq \sigma^{t}{(m-1)}$.
Suppose $\sigma^{(s)}{(l-1)} \equiv \sigma^{(t)}{(l-1)}  \Rightarrow s\equiv t$ $\pmod {n-1}$, which is a contradiction.
Suppose $\sigma^{(t)}{(l-1)} \equiv \sigma^{(t)}{(m-1)} \Rightarrow l\equiv m$ $\pmod {n-1}$, which is a contradiction.
Suppose $\sigma^{(s)}{(l-1)} \equiv \sigma^{(t)}{(m-1)} \Rightarrow s+l\equiv t+m$ $\pmod {n-1}$, and assume there are three distinct elements:
$\sigma^{t}{(l-1)}\neq \sigma^{t}{(m-1)}$,
$\sigma^{s}{(m-1)}\neq \sigma^{t}{(m-1)}$, 
$\sigma^{s}{(m-1)}\neq \sigma^{t}{(l-1)}$. 
Follow the argument above the first two inequalities are correct.
Suppose $\sigma^{s}{(m-1)} \equiv \sigma^{t}{(l-1)}\Rightarrow s+m\equiv t+l$ $\pmod {n-1}$.
Subtracting equations $s+l \equiv t+m$ and  $s+m\equiv t+l\Rightarrow l \equiv m$ $\pmod {n-1}$, which is a contradiction.
One of the following two sets has three distinct elements:
$\{ \sigma^{s}{(l-1)},\;   \sigma^{t}{(l-1)},\;   \sigma^{t}{(m-1)} \}$ or $\{ \sigma^{t}{(l-1)},\;   \sigma^{t}{(m-1)},\;   \sigma^{s}{(m-1)} \}$. 
  
\vspace{.1 in}

\emph{\textbf{Case 2}}:
We take the submatrix of $G(i,j;l,m)$ with $2 \leq l< m \leq n$, $i=1$, $1<j<n$, and let $r=(n+1)-j$. 
A typical ($2\times 2$) submatrix has the form:

$$\begin{pmatrix}
                    l&m\cr
                    \sigma^{r}{(l-1)}& \sigma^{r}{(m-1)}\end{pmatrix}$$

We wish to show there are three distinct elements: 
$l \neq m$,
$m \neq \sigma^{r}{(m-1})$, 
$l \neq \sigma^{r}{(m-1)}$.  
Suppose $m \equiv \sigma^{(r)}{(m-1)}\Rightarrow r \equiv 1$ $\pmod {n-1}$, which is a contradiction.
Suppose $l \equiv \sigma^{(r)}{(m-1)}\Rightarrow l \equiv r+m-1$ $\pmod {n-1}$, and assume there are three distinct elements: 
$\sigma^{r}{(l-1)}\neq \sigma^{r}{(m-1)}$,
$m\neq \sigma^{r}{(m-1)}$,
$m\neq \sigma^{r}{(l-1)}$.  
The first two inequalities are correct. Suppose $m \equiv \sigma^{r}{(l-1)}\Rightarrow m \equiv r+l-1$ $\pmod {n-1}$. 
Subtracting equations $l \equiv r+m-1$ and  $m \equiv r+l-1\Rightarrow r=1$ $\pmod {n-1}$, which is a contradiction.
One of the following two sets has three distinct elements:
$\{l,\;  m,\;  \sigma^{r}{(m-1)} \}$ or $\{ \sigma^{r}{(l-1)},\; \sigma^{r}{(m-1)},\;  m \}$. 
 
\vspace{.1 in}
 
\noindent \emph{\textbf{Step 2}}:

For matrix $G(i,j;l,m)$ with  $i=1$, $j=n$ and $2 \leq l <m\leq n$ we look at five cases and consider every matrix type defined above of even $n\pmod 6$.

\vspace{.1 in}

\emph{\textbf{Case 1}}: We take $G(i,j;l,m)$ with   $2\leq l<m \leq \frac{n}{2}-1$, $i=1$, $j=n$. 

\vspace{.1in}
\noindent \textit{Consider  $G_1$}.

  $$\begin{pmatrix}
       l&m\cr
       n-2l&n-2m\end{pmatrix}$$
                   
We wish to show there are three distinct entries:
$l\neq n-2l$,
$n-2l\neq n-2m$, 
$l\neq n-2m$. 
Suppose $l =n-2l\Rightarrow 3l=n$ and since $n=2+6k$  this is a contradiction.
Suppose $n-2l=n-2m \Rightarrow l=m$, which is a contradiction.
Suppose $l=n-2m$ and we wish to show there are three distinct elements:
$l\neq m$,
$l\neq n-2l$,
$m\neq n-2l$.
As shown above the first inequality is correct.
Suppose $m=n-2l$ and since $l=n-2m \Rightarrow l=m$, which is a contradiction.
One of the following two sets has three distinct elements: $\{l, n-2l, n-2m\}$ or $\{l, m, n-2l\}$. 

\vspace{.1in}

\noindent \textit{Consider $G_2$}

{\bf 1.}  $G_2$ with  $2 \leq l< m \leq \frac{n}{2}-2$, $i=1$, $j=n$.

$$\begin{pmatrix}
                   l&m\cr
                   n-2(l+1)&n-2(m+1)\end{pmatrix}$$

We wish to show there are three distinct entries:
$l\neq n-2(m+1)$,
$l\neq n-2(l+1)$, 
$n-2(l+1)\neq n-2(m+1)$.
Suppose $l =n-2l-2\Rightarrow 3l=n-2$ and since $n=6+6k$  this is a contradiction.
Suppose $n-2l=n-2m \Rightarrow l=m$, which is a contradiction.
Suppose $l=n-2(m+1)$ and we wish to show there are three distinct elements:
$l\neq m$,
$l\neq n-2(l+1)$,
$m\neq n-2(l+1)$.
As shown above the first inequality is correct.
Suppose $m=n-2(l+1)$ and since $l=n-2(m+1) \Rightarrow l=m$, which is a contradiction.
One of the following two sets has three distinct elements: $\{l, n-2(m+1), n-2(l+1)\}$ or $\{l, m, n-2(l+1)\}$. 

{\bf 2.} $G_2$ with $ 2 \leq l \leq \frac{n}{2}-2$, $m=\frac{n}{2}-1$, $i=1$, $j=n$.

$$\begin{pmatrix}
                   l&m\cr
                   n-2(l+1)&Y\end{pmatrix}$$

\noindent If  \textit{$K$ is Even}  $\Rightarrow m=\frac{n}{2}-1$, $Y=\frac{n}{2}-2 \Rightarrow Y=m-1$.

We wish to show there are three distinct entries:
$l\neq m$,
$m\neq m-1$, 
$l\neq m-1$.
Assume $l= m-1$ and we wish to show there are three distinct entries:
$m\neq m-1$,
$m\neq n-2(l+1)$,
$m-1\neq n-2(l+1)$.
Suppose $m= n-2(l+1)$ and since $l=m-1$ and $m=\frac{n}{2}-1 \Rightarrow  n=6$, which is a contradiction.
Suppose $m-1= n-2(l+1)$ and since $m=\frac{n}{2}-1$ and $l= m-1 \Rightarrow  n=8$, which is a contradiction.
One of the following two sets has three distinct elements: $\{l, m, Y\}$ or $\{m, Y, n-2(l+1)\}$. 
\vspace{.1in}

\noindent If  \textit{$K$ is Odd} $ \Rightarrow m=\frac{n}{2}-1$, $Y=\frac{n}{2}+1 \Rightarrow  Y=m+2$.
Three distinct elements are $\{l, m, Y\}$. 

\vspace{.1in}

\noindent \textit{Consider $G_3$}

{\bf 1.}  $G_3$ with $i=1$, $j=n$ and ($2 \leq l<m \leq \frac{n-4}{6}$ or $\frac{n+2}{6} \leq l<m \leq \frac{n}{2}-2$). The argument is similar to above one with $G_1$. One of the following two sets has three distinct elements: $\{l, n-2l, n-2m \}$ or $\{ l, m, n-2l \}$. 

{\bf2.} $G_3$ with   $2 \leq l \leq \frac{n-4}{6}$, $\frac{n+2}{6} \leq m \leq \frac{n}{2}-2$, $i=1$, $j=n$.

$$\begin{pmatrix}
                   l&m\cr
                   n-2l&n-2(m+1)\end{pmatrix}$$

We wish to show there are three distinct entries:
$l\neq n-2l$,
$n-2l\neq n-2(m+1)$, 
$l\neq n-2(m+1)$.
Suppose $l =n-2l\Rightarrow 3l=n$ and since $n=4+6k$ this is a contradiction.
Suppose $n-2l= n-2(m+1)\Rightarrow l=m+1$, which is a contradiction.
Suppose $l=n-2(m+1)$ and we wish to show there are three distinct elements:
$l \neq m$,
$l\neq n-2l$,
$m \neq n-2l$.
The first two inequalities are correct.
Suppose $m = n-2l$ and since  $l=n-2(m+1)  \Rightarrow m=l-2$, which is a contradiction.
One of the following two sets has three distinct elements: $\{l, n-2l, n-2(m+1))\}$ or $\{ l, m, n-2l \}$. 

{\bf3.}  $G_3$ with $2 \leq l \leq \frac{n-4}{6}$, $m=\frac{n}{2}-1$, $i=1$, $j=n$.

 $$\begin{pmatrix}
                   l&m\cr
                   n-2l&n-9\end{pmatrix}$$

We wish to show  there are three distinct entries:
$l\neq m$,
$m\neq n-9$,
$l\neq n-9$.
Suppose $l= n-9$ and since  $l <\frac{n-4}{6}\Rightarrow n-9<\frac{n-4}{6}\Rightarrow n<10$, which is a contradiction.
Suppose $m =n-9\Rightarrow \frac{n}{2}-1=n-9\Rightarrow n=16$, which is a contradiction.
There are three distinct elements $\{l, m, n-9\}$. 

{\bf  4.}  $G_3$ with  $\frac{n+2}{6}\leq l \leq \frac{n}{2}-2$, $m=\frac{n}{2}-1$, $i=1$, $j=n$.

 $$\begin{pmatrix}
                   l&m\cr
                   n-2(l+1)&n-9\end{pmatrix}$$

We wish to show  there are three distinct entries:
$l\neq m$,
$m\neq n-9$,
$l\neq n-9$.
Suppose $l= n-9$ and since  $l <\frac{n}{2}-2\Rightarrow n-9<\frac{n}{2}-2\Rightarrow n<14$, which is a contradiction.
Suppose $m =n-9\Rightarrow \frac{n}{2}-1=n-9\Rightarrow n=16$, which is a contradiction.
There are three distinct elements $\{l, m, n-9\}$. 

\vspace{.1in}

\emph{\textbf{Case 2}}:  For the submatrix $G(i,j;l,m)$ with $i=1$, $j=n$ and ($\frac{n}{2} \leq l < m \leq n-2$ or $2\leq l \leq \frac{n}{2}-1$, $\frac{n}{2} \leq  m \leq n-2$) three distinct elements are $\{l, m, n\}$.

\vspace{.1in}

\emph{\textbf{Case 3}}: We take the submatrix $G(i,j;l,m)$ with  $2\leq l \leq \frac{n}{2}-1$, $ m = n-1$, $i=1$, $j=n$.
\vspace{.1in}

\noindent \textit{Consider $G_1$}

$$\begin{pmatrix}
                   l&m\cr
                   n-2l&n-1\end{pmatrix}$$

We wish to show there are three distinct entries:
$l \neq n-1$,
$n-2l \neq n-1$,
$l \neq n-2l$.
Suppose  $n-2l = n-1\Rightarrow l=\frac{1}{2}$, which is a contradiction.
Suppose  $l=n-2l\Rightarrow 3l=n$ and since $n=2+6k$ this is a contradiction.
There are three distinct elements $\{l, n-1, n-2l\}$.
\vspace{.1in}

\noindent \textit{Consider $G_2$}

{\bf 1.} $G_2$ with $2\leq l \leq \frac{n}{2}-2$, $m =n-1$, $i=1$, $j=n$.

$$\begin{pmatrix}
                   l&m\cr
                   n-2(l+1)&n-1\end{pmatrix}$$

We wish to show there are three distinct entries:                            
$l \neq n-1$,
$n-2(l+1) \neq n-1$,
$l \neq n-2(l+1)$.
Suppose  $n-2(l+1) = n-1 \Rightarrow l=-\frac{1}{2}$, which is a contradiction.
Suppose $l=n-2(l+1)\Rightarrow 3l=n-2$ and since $n=6+6k$   this is a contradiction.
There are three distinct elements $\{l, n-1, n-2(l+1)\}$.

{\bf 2.}  $G_2$ with  $l =\frac{n}{2}-1$, $m=n-1$, $i=1$, $j=n$.

$$\begin{pmatrix}
                   l&m\cr
                   Y&n-1\end{pmatrix}$$

If  \textit{$K$ is Even} $\Rightarrow Y=\frac{n}{2}-2$.                         
If  $n-1 = Y\Rightarrow n-1=\frac{n}{2}-2\Rightarrow n=-2$, which is a contradiction. Three distinct entries are $\{l,Y,n-1\}$.

If  \textit{$K$ is Odd} $\Rightarrow Y=\frac{n}{2}+1$, and three distinct entries are $\{l,Y,n-1\}$.

\vspace{.1in}

\noindent \textit{Consider $G_3$}

{\bf 1.} For $G_3$ with $l \leq \frac{n-4}{6}$, $ m = n-1$, $i=1$, $j=n$ three distinct elements are $\{l, n-1, n-2l\}$ (similar to $G_1$.)

{\bf 2.} $G_3$ with $\frac{n+2}{6} \leq l < n-1$, $m = n-1$, $i=1$, $j=n$.

$$\begin{pmatrix}
                   l&m\cr
                   n-2(l+1)&n-1\end{pmatrix}$$

We wish to show there are three distinct entries:        
$l \neq n-1$,
$n-2(l+1)\neq n-1$,
$l\neq  n-2(l+1)$.
Suppose  $n-2(l+1)= n-1\Rightarrow l=-\frac{1}{2}$, which is a contradiction.
Suppose $l= n-2(l+1)\Rightarrow 3l=n-2$ and since $n=4+6k$ this is a contradiction.
Three distinct elements are $\{l, n-1, n-2(l+1)\}$. 

{\bf 3.} For $G_3$ with $l=\frac{n}{2}-1$, $m=n-1$, $i=1$, $j=n$ three distinct entries are $\{l, n-9, n-1)\}$.

\vspace{.1in}

\emph{\textbf{Case 4}}: For the submatrix $G(i,j;l,m)$ with   $\frac{n}{2} \leq l \leq  n-2$, $ m=n-1$, $i=1$, $j=n$
three distinct elements are $\{l,n,n-1\}$.

\vspace{.1in}

\emph{\textbf{Case 5}}: For the submatrix $G(i,j;l,m)$ with  $l = n-1$, $ m=n$, $i=1$, $j=n$ three distinct elements are $\{n-1,n,1\}$.

 \qed\end{proof}
  
The argument for other steps is similar. To see the details please view the appendix to this article on ArXiv at http://arxiv.org/ or contact the first author.

 \bibliographystyle{plain}

 \end{document}